\documentclass{amsart}
\usepackage{amsmath}
\usepackage{amssymb}
\usepackage{amsfonts}
\usepackage{hyperref}

\newtheorem{theorem}{Theorem}[section]
\theoremstyle{plain}

\newtheorem{corollary}{Corollary}[section]

\newtheorem{definition}{Definition}

\newtheorem{lemma}{Lemma}[section]

\newtheorem{remark}{Remark}

\numberwithin{equation}{section}
\allowdisplaybreaks[1]

\begin{document}
\title[New integral inequalities via $(\alpha,m)$-convexity and quasi-convexity ]{New integral inequalities via $(\alpha,m)$-convexity and quasi-convexity}
\author[W. J. Liu]{Wenjun Liu}
\address[W. J. Liu]{College of Mathematics and Statistics\\
Nanjing University of Information Science and Technology \\
Nanjing 210044, China}
\email{wjliu@nuist.edu.cn}


\begin{abstract}
In this paper, we establish some new integral inequalities for $(\alpha, m)-$convex functions and quasi-convex functions, respectively. Our
results in special cases recapture known results.
\end{abstract}

\maketitle

\section{INTRODUCTION}

Let $I$ be on interval in $R$. Then $f:I\rightarrow R$ is said to be convex (see \cite[P.1]{MPF})
if 
\begin{equation*}
f\left( t x+\left( 1-t \right) y\right) \leq t f\left(
x\right) +\left( 1-t \right) f\left( y\right)
\end{equation*}%
holds for all $x,y\in I$ and $t \in \left[ 0,1\right]$.

In \cite{GT1}, Toader defined $m$-convexity as follows:

\begin{definition}
\label{d1} The function $f:\left[ 0,b\right] \rightarrow R$, $b>0$ is said
to be $m$-convex, where $m\in \left[ 0,1\right] $, if
\begin{equation*}
f\left( tx+m\left( 1-t\right) y\right) \leq tf\left( x\right) +m\left(
1-t\right) f\left( y\right)
\end{equation*}%
holds for all $x,y\in \left[ 0,b\right] $ and $t\in \left[ 0,1\right] .$We say
that $f$ is $m-$concave if $-f$ is $m-$convex.
\end{definition}

%Denote by $K_{m}\left( b\right) $ the class of all $m-$convex functions on $%
%\left[ 0,b\right] $ for which $f\left( 0\right) \leq 0.$ Obviously, if we
%choose $m=1,$ Definition  \ref{d1}  recaptures the concept of standard
%convex functions on $\left[ 0,b\right] .$

In \cite{VGM},   Mihe\c{s}an defined $\left( \alpha, m\right) -$
convexity as  follows:

\begin{definition}
\label{d.1.2} The function $f:\left[ 0,b\right] \rightarrow
%TCIMACRO{\U{211d} }%
%BeginExpansion
\mathbb{R}
%EndExpansion
$,  $b>0$, is said to be $\left( \alpha, m\right) -$ convex, where $\left(
\alpha, m\right) \in \left[ 0,1\right] ^{2}$, if  
\begin{equation*}
f(tx+m(1-t)y)\leq t^{\alpha }f(x)+m(1-t^{\alpha })f(y)
\end{equation*}%
holds for all $x,y\in \left[ 0,b\right] $ and $t\in \left[ 0,1\right] $.
\end{definition}

Denote by $K_{m}^{\alpha }(b)$ the class of all $\left( \alpha, m\right) -$%
convex functions on $\left[ 0,b\right] $ for which $f(0)\leq 0$. It can be
easily seen that for $\left( \alpha, m\right) =\left( 1,m\right), $ $\left(
\alpha, m\right) -$ convexity reduces to $m-$ convexity and for $\left(
\alpha, m\right) =\left( 1,1\right) $, $\left( \alpha, m\right) -$ convexity
reduces to the concept of usual convexity defined on $\left[ 0,b\right] $,  $%
b>0$. For recent results and generalizations
concerning $m-$convex and $\left( \alpha, m\right)-$convex functions see
\cite{BOP, BPR, Dragomir, Ozdemir1, OSS, SSOR}.

We recall that the notion of quasi-convex functions generalizes the notion
of convex functions. More precisely, a function $f:[a,b]\rightarrow \mathbb{R}$ is said to be quasi-convex on $[a,b]$ if%
\begin{equation*}
f(\lambda x+(1-\lambda )y)\leq \max \{f(x),f(y)\}
\end{equation*}%
holds for any $x,y\in \lbrack a,b]$ and $\lambda \in \lbrack 0,1]$. Clearly, any
convex function is a quasi-convex function. Furthermore, there exist
quasi-convex functions which are not convex (see \cite{daI}).

One of the most famous inequalities for convex functions is Hadamard's
inequality. This double inequality is stated as follows: Let $f$ be a convex function on some
nonempty interval $[a,b]$ of real line $\mathbb{R}$, where $a\neq b$. Then
\begin{equation}
f\left( \frac{a+b}{2}\right) \leq \frac{1}{b-a}\int_{a}^{b}f(x)dx\leq \frac{%
f\left( a\right) +f\left( b\right) }{2}.  \label{1.1}
\end{equation}
Hadamard's
inequality for convex functions has received renewed attention in recent
years and a remarkable variety of refinements and generalizations have been
found (see, for example, \cite{Alomari}-\cite{Ozdemir1}, \cite{Pecaric}-\cite{SSOR}, \cite{THD}). In \cite{BOP}, Bakula et al. establish several Hadamard type inequalities for differentiable $m-$convex and $\left( \alpha, m\right)-$convex functions.

Recently, Ion \cite{daI} established two estimates on the Hermite-Hadamard inequality for functions
whose first derivatives in absolute value are quasi-convex. Namely, he
obtained the following results:

\begin{theorem}\label{th1.1}
Let $f:I\subset\mathbb{R}\rightarrow\mathbb{R}$ be a differentiable mapping on $I$, $a,b\in I$ with $a<b.$ If $%
\left\vert f^{\prime }\right\vert $ is quasi-convex on $[a,b]$, then the
following inequality holds:%
\begin{equation*}
\left\vert \frac{f(a)+f(b)}{2}-\frac{1}{b-a}\int_{a}^{b}f(u)du\right\vert
\leq \frac{b-a}{4}\left\{ \max \left\vert f^{\prime }(a)\right\vert
,\left\vert f^{\prime }(b)\right\vert \right\} .
\end{equation*}
\end{theorem}

\begin{theorem}\label{th1.2}
Let $f:I\subset\mathbb{R}\rightarrow\mathbb{R}$ be a differentiable mapping on $I$, $a,b\in I$ with $a<b$ and let $p>1$. If $%
\left\vert f^{\prime }\right\vert ^{\frac{p}{p-1}}$ is quasi-convex on $%
[a,b],$ then the following inequality holds:%
\begin{equation*}
\left\vert \frac{f(a)+f(b)}{2}-\frac{1}{b-a}\int_{a}^{b}f(u)du\right\vert
\leq \frac{b-a}{2(p+1)^{\frac{1}{p}}}\left( \max \left\{ \left\vert
f^{\prime }(a)\right\vert ^{\frac{p}{p-1}},\left\vert f^{\prime
}(b)\right\vert ^{\frac{p}{p-1}}\right\} \right) ^{\frac{p-1}{p}}.
\end{equation*}
\end{theorem}

In \cite{ADD}, Alomari et al. obtained the following result.

\begin{theorem}\label{th1.3}
Let $f:I\subset\mathbb{R}\rightarrow\mathbb{R}$ be a differentiable mapping on $I$, $a,b\in I$ with $a<b$ and let $q\ge 1$. If $%
\left\vert f^{\prime }\right\vert ^{q}$ is quasi-convex on $[a,b],$ then the following inequality holds:%
\begin{equation*}
\left\vert \frac{f(a)+f(b)}{2}-\frac{1}{b-a}\int_{a}^{b}f(u)du\right\vert
\leq \frac{b-a}{4}\left( \max \left\{ \left\vert
f^{\prime }(a)\right\vert ^{q},\left\vert f^{\prime
}(b)\right\vert ^{q}\right\} \right) ^{\frac{1}{q}}.
\end{equation*}
\end{theorem}

In \cite{OSA},  \"{O}zdemir et al. used the
following lemma in order to establish several integral inequalities via some kinds of convexity.

\begin{lemma}\label{le1.1} Let $f: [a,b]\subset [0, \infty)\rightarrow\mathbb{R}$ be continuous on $[a,b]$ such that $f \in L([a,b])$, $a<b$. Then the equality
\begin{equation}\label{1.2}
\int_a^b (x-a)^p(b-x)^qf(x)dx=(b-a)^{p+q+1} \int_0^1 (1-t)^pt^qf(ta+(1-t)b)dt
\end{equation}
holds for some fixed $p, q > 0$.
\end{lemma}

Especially, \"{O}zdemir et al. \cite{OSA} discussed the following new
results connecting with $m-$convex function and quasi-convex function, respectively:
\begin{theorem}\label{th1.4}
Let $f: [a,b]\rightarrow\mathbb{R}$ be continuous on $[a,b]$ such that $f \in L([a,b])$, $0\le a<b<\infty$. If $f$ is $m-$convex on $[a,b]$,  for some fixed $m\in (0, 1]$ and $p, q > 0$, then
\begin{align}\label{1.3}
&\int_a^b (x-a)^p(b-x)^qf(x)dx\nonumber\\
\leq & (b-a)^{p+q+1} \min\left\{\beta(q+2, p+1)f(a)+m\beta(q+1, p+2)f\left(\frac{b}{m}\right),\right.\nonumber\\
 & \left.\beta(q+1, p+2)f(b)+m\beta(q+2, p+1)f\left(\frac{a}{m}\right)
\right\},
\end{align}
where $\beta (x, y)$ is the Euler Beta function.
\end{theorem}

\begin{theorem}\label{th1.5}
Let $f: [a,b]\rightarrow\mathbb{R}$ be continuous on $[a,b]$ such that $f \in L([a,b])$, $0\le a<b<\infty$. If $f$ is quasi-convex on $[a,b]$, then for some fixed   $p, q > 0$, we have
\begin{align}\label{1.4}
 \int_a^b (x-a)^p(b-x)^qf(x)dx
\leq (b-a)^{p+q+1} \max\{f(a), f(b)\}\beta(p+1, q+1).
\end{align}
\end{theorem}

The aim of this paper is to establish some new integral inequalities like those given in Theorems \ref{th1.4} and   \ref{th1.5} for $(\alpha, m)-$convex functions (Section \ref{se2}) and quasi-convex functions (Section \ref{se3}), respectively. Our
results in special cases recapture Theorems \ref{th1.4} and   \ref{th1.5}, respectively. That is, this study
is a continuation and generalization of \cite{OSA}.

\section{New integral inequalities for $(\alpha, m)-$ convex functions}\label{se2}

\begin{theorem}\label{th2.1}
Let $f: [a,b]\rightarrow\mathbb{R}$ be continuous on $[a,b]$ such that $f \in L([a,b])$, $0\le a<b<\infty$. If $f$ is $(\alpha, m)-$convex on $[a,b]$,  for some fixed $(\alpha, m) \in \left( 0,1\right] ^{2}$ and $p, q > 0$, then
\begin{align}\label{2.1}
&\int_a^b (x-a)^p(b-x)^qf(x)dx\nonumber\\
\leq & (b-a)^{p+q+1}\min\left\{\beta(q+\alpha+1, p+1)f(a) +m[\beta(q+1, p+1)-\beta(q+\alpha+1, p+1)]f\left(\frac{b}{m}\right),\right.\nonumber\\
 & \left.\beta(q+1, p+\alpha+1)f(b)+m[\beta(p+1, q+1)-\beta(q+1, p+\alpha+1)]f\left(\frac{a}{m}\right)
\right\},
\end{align}
where $\beta (x, y)$ is the Euler Beta function.
\end{theorem}

\begin{proof}
Since $f$ is $(\alpha, m)-$convex on $[a,b]$, we know that for every $t\in [0,1]$
\begin{equation}
f(ta+(1-t)b)=f\left(ta+m(1-t)\frac{b}{m}\right)\leq  t^\alpha f(a)+m\left(1-t^\alpha\right)f\left(\frac{b}{m}\right). \label{2.2}
\end{equation}
Using Lemma \ref{le1.1}, with $x=ta+(1-t)b$, then we have
\begin{align*}
&\int_a^b (x-a)^p(b-x)^qf(x)dx\\
\leq & (b-a)^{p+q+1}\int_0^1 (1-t)^pt^q\left(t^\alpha f(a)+m\left(1-t^\alpha\right)f\left(\frac{b}{m}\right)\right)dt\\
= & (b-a)^{p+q+1}\left[f(a)\int_0^1 (1-t)^pt^{q+\alpha}dt +m f\left(\frac{b}{m}\right) \int_0^1 (1-t)^pt^{q}\left(1-t^\alpha\right)dt \right].
\end{align*}
Now, we will make use of the Beta function which is defined for $x, y > 0$ as
$$\beta (x, y)=\int_0^1 t^{x-1}(1-t)^{y-1}dt.$$
It is known that
$$\int_0^1t^{q+\alpha} (1-t)^pdt=\beta (q+\alpha+1, p+1),$$
\begin{align*}\int_0^1 (1-t)^pt^{q}\left(1-t^\alpha\right)dt=&\int_0^1 t^{q}(1-t)^pdt-\int_0^1 t^{q+\alpha}(1-t)^pdt\\
=&\beta(q+1, p+1)-\beta(q+\alpha+1, p+1)].\end{align*}
Combining all obtained equalities we get
\begin{align}\label{2.3}
&\int_a^b (x-a)^p(b-x)^qf(x)dx\nonumber\\
\leq & (b-a)^{p+q+1} \left\{\beta(q+\alpha+1, p+1)f(a)+m[\beta(q+1, p+1)-\beta(q+\alpha+1, p+1)]f\left(\frac{b}{m}\right)
\right\}.
\end{align}
If we choose $x = tb + (1-t)a$, analogously we obtain
\begin{align}\label{2.4}
&\int_a^b (x-a)^p(b-x)^qf(x)dx\nonumber\\
\leq & (b-a)^{p+q+1} \left\{\beta(q+1, p+\alpha+1)f(b) +m[\beta(q+1, p+1)-\beta(q+1, p+\alpha+1)]f\left(\frac{a}{m}\right)
\right\}.
\end{align}
Thus, by \eqref{2.3} and \eqref{2.4} we obtain \eqref{2.1}, which completes the
proof.
\end{proof}

\begin{remark}
\label{re1} As a special case of Theorem \ref{th2.1} for $\alpha =1,$ that is for $f$ be $m-$convex on $[a,b]$, we
recapture Theorem \ref{th1.4} due to the fact that
\begin{align*}
\beta(q+1, p+1)-\beta(q+2, p+1)=&\beta(q+1, p+1)-\frac{q+1}{p+q+2}\beta(q+1, p+1)\\
=&\frac{p+1}{p+q+2}\beta(q+1, p+1)=\beta(q+1, p+2)
\end{align*}
and
$$\beta(q+1, p+1)-\beta(q+1, p+\alpha+1)=\beta(q+2, p+1).$$
\end{remark}

\begin{corollary}
\label{co2.1} In Theorem \ref{th2.1}, if $p=q$, then \eqref{2.1} reduces to
\begin{align*}
&\int_a^b (x-a)^p(b-x)^pf(x)dx\nonumber\\
\leq & (b-a)^{2p+1}\min\left\{\beta(p+\alpha+1, p+1)f(a) +m[\beta(p+1, p+1)-\beta(p+\alpha+1, p+1)]f\left(\frac{b}{m}\right),\right.\nonumber\\
 & \left.\beta(p+1, p+\alpha+1)f(b)+m[\beta(p+1, p+1)-\beta(p+1, p+\alpha+1)]f\left(\frac{a}{m}\right)
\right\}.
\end{align*}
\end{corollary}

\begin{theorem}\label{th2.2}
Let $f: [a,b]\rightarrow\mathbb{R}$ be continuous on $[a,b]$ such that $f \in L([a,b])$, $0\le a<b<\infty$ and let $k>1$. If $|f|^{\frac{k}{k-1}}$ is $(\alpha, m)-$convex on $[a,b]$,  for some fixed $(\alpha, m) \in \left( 0,1\right] ^{2}$ and $p, q > 0$, then
\begin{align}\label{2.5}
&\int_a^b (x-a)^p(b-x)^qf(x)dx\nonumber\\
\leq & \frac{(b-a)^{p+q+1}}{(\alpha+1)^\frac{k-1}{k}}\left[\beta(kp+1, kq+1)\right]^{\frac{1}{k}}\min\left\{\left[|f(a)|^\frac{k}{k-1}+\alpha m\left|f\left(\frac{b}{m}\right)\right|^\frac{k}{k-1}\right]^\frac{k-1}{k},\right.\nonumber\\
 & \left. \left[|f(b)|^\frac{k}{k-1}+\alpha m\left|f\left(\frac{a}{m}\right)\right|^\frac{k}{k-1}\right]^\frac{k-1}{k}
\right\}.
\end{align}
\end{theorem}

\begin{proof}
Since $|f|^{\frac{k}{k-1}}$  is $(\alpha, m)-$convex on $[a,b]$ we know that for every $t\in [0,1]$
\begin{align*}
|f(ta+(1-t)b)|^{\frac{k}{k-1}}=&\left|f\left(ta+m(1-t)\frac{b}{m}\right)\right|^{\frac{k}{k-1}}\\
\leq & t^\alpha |f(a)|^{\frac{k}{k-1}}+m\left(1-t^\alpha\right)\left|f\left(\frac{b}{m}\right)\right|^{\frac{k}{k-1}}.
\end{align*}
Using Lemma \ref{le1.1}, with $x=ta+(1-t)b$, then we have
\begin{align*}
&\int_a^b (x-a)^p(b-x)^qf(x)dx\\
\leq & (b-a)^{p+q+1}\left[\int_0^1 (1-t)^{kp}t^{kq}dt\right]^\frac{1}{k}
\left[\int_0^1 |f(ta+(1-t)b)|^{\frac{k}{k-1}} dt\right]^\frac{k-1}{k} \\
\leq & (b-a)^{p+q+1}\left[\beta(kq+1, kp+1)\right]^{\frac{1}{k}} \left[ \int_0^1 t^\alpha |f(a)|^{\frac{k}{k-1}}dt
+m \int_0^1 \left(1-t^\alpha\right)\left|f\left(\frac{b}{m}\right)\right|^{\frac{k}{k-1}} dt \right]^\frac{k-1}{k}\\
= & (b-a)^{p+q+1}\left[\beta(kq+1, kp+1)\right]^{\frac{1}{k}} \left[  \frac{1}{\alpha+1} |f(a)|^{\frac{k}{k-1}}
+m \frac{\alpha}{\alpha+1} \left|f\left(\frac{b}{m}\right)\right|^{\frac{k}{k-1}}  \right]^\frac{k-1}{k}.
\end{align*}
If we choose $x = tb + (1-t)a$, analogously we obtain
\begin{align*}
&\int_a^b (x-a)^p(b-x)^qf(x)dx\\
\leq & (b-a)^{p+q+1}\left[\beta(kp+1, kq+1)\right]^{\frac{1}{k}} \left[  \frac{1}{\alpha+1} |f(b)|^{\frac{k}{k-1}}
+m \frac{\alpha}{\alpha+1} \left|f\left(\frac{a}{m}\right)\right|^{\frac{k}{k-1}}  \right]^\frac{k-1}{k},
\end{align*}
which completes the proof.
\end{proof}

\begin{corollary}
\label{co2.2} In Theorem \ref{th2.2}, if $p=q$, then \eqref{2.5} reduces to
\begin{align*}
&\int_a^b (x-a)^p(b-x)^pf(x)dx\nonumber\\
\leq & \frac{(b-a)^{2p+1}}{(\alpha+1)^\frac{k-1}{k}}\left[\beta(kp+1, kp+1)\right]^{\frac{1}{k}}\min\left\{\left[|f(a)|^\frac{k}{k-1}+\alpha m\left|f\left(\frac{b}{m}\right)\right|^\frac{k}{k-1}\right]^\frac{k-1}{k},\right.\nonumber\\
 & \left. \left[|f(b)|^\frac{k}{k-1}+\alpha m\left|f\left(\frac{a}{m}\right)\right|^\frac{k}{k-1}\right]^\frac{k-1}{k}
\right\}.
\end{align*}
\end{corollary}

\begin{corollary}
\label{co2.3} In Theorem \ref{th2.2}, if $\alpha =1,$ i.e., if $|f|^{\frac{k}{k-1}}$ is $m-$convex on $[a,b]$, then \eqref{2.5} reduces to
\begin{align*}
&\int_a^b (x-a)^p(b-x)^qf(x)dx\nonumber\\
\leq & \frac{(b-a)^{p+q+1}}{2^\frac{k-1}{k}}\left[\beta(kp+1, kq+1)\right]^{\frac{1}{k}}\min\left\{\left[|f(a)|^\frac{k}{k-1}+ m\left|f\left(\frac{b}{m}\right)\right|^\frac{k}{k-1}\right]^\frac{k-1}{k},\right.\nonumber\\
 & \left. \left[|f(b)|^\frac{k}{k-1}+ m\left|f\left(\frac{a}{m}\right)\right|^\frac{k}{k-1}\right]^\frac{k-1}{k}
\right\}.
\end{align*}
\end{corollary}

\begin{remark}
\label{re2} As a special case of Corollary \ref{co2.3} for $m =1,$ that is for $|f|^{\frac{k}{k-1}}$ be convex on $[a,b]$, we
get
\begin{align*}
 \int_a^b (x-a)^p(b-x)^qf(x)dx 
\leq   \frac{(b-a)^{p+q+1}}{2^\frac{k-1}{k}}\left[\beta(kp+1, kq+1)\right]^{\frac{1}{k}} \left[|f(a)|^\frac{k}{k-1}+ \left|f\left(b\right)\right|^\frac{k}{k-1}\right]^\frac{k-1}{k}.
\end{align*}
\end{remark}

\begin{theorem}\label{th2.3}
Let $f: [a,b]\rightarrow\mathbb{R}$ be continuous on $[a,b]$ such that $f \in L([a,b])$, $0\le a<b<\infty$ and let $l\ge 1$. If $|f|^{l}$ is $(\alpha, m)-$convex on $[a,b]$,  for some fixed $(\alpha, m) \in \left( 0,1\right] ^{2}$ and $p, q > 0$, then
\begin{align}\label{2.6}
&\int_a^b (x-a)^p(b-x)^qf(x)dx\nonumber\\
\leq & (b-a)^{p+q+1}\left[\beta(p+1, q+1)\right]^{\frac{l-1}{l}}\nonumber\\
&\times\min\left\{\left[ \beta(q+\alpha+1, p+1)|f(a)|^l +m[\beta(q+1, p+1)-\beta(q+\alpha+1, p+1)]\left|f\left(\frac{b}{m}\right) \right|^l \right]^\frac{1}{l},\right.\nonumber\\
 & \left. \left[ \beta(q+1, p+\alpha+1)|f(b)|^l +m[\beta(q+1, p+1)-\beta(q+1, p+\alpha+1)]\left|f\left(\frac{a}{m}\right) \right|^l \right]^\frac{1}{l}
\right\}.
\end{align}
\end{theorem}

\begin{proof}
Since $|f|^{l}$  is $(\alpha, m)-$convex on $[a,b]$, we know that for every $t\in [0,1]$
\begin{align*}
|f(ta+(1-t)b)|^{l}= \left|f\left(ta+m(1-t)\frac{b}{m}\right)\right|^{l}
\leq   t^\alpha |f(a)|^{l}+m\left(1-t^\alpha\right)\left|f\left(\frac{b}{m}\right)\right|^{l}.
\end{align*}
Using Lemma \ref{le1.1}, with $x=ta+(1-t)b$, then we have
\begin{align*}
&\int_a^b (x-a)^p(b-x)^qf(x)dx\\
=&(b-a)^{p+q+1} \int_0^1 \left[(1-t)^pt^q\right]^\frac{l-1}{l} \left[(1-t)^pt^q\right]^\frac{1}{l} f(ta+(1-t)b)dt\\
\leq & (b-a)^{p+q+1}\left[\int_0^1 (1-t)^{p}t^{q}dt\right]^\frac{l-1}{l}
\left[\int_0^1 (1-t)^{p}t^{q} |f(ta+(1-t)b)|^{l} dt\right]^\frac{1}{l} \\
\leq & (b-a)^{p+q+1}\left[\beta(q+1, p+1)\right]^{\frac{l-1}{l}} \\&
\times\left[ \beta(q+\alpha+1, p+1)|f(a)|^l +m[\beta(q+1, p+1)-\beta(q+\alpha+1, p+1)]\left|f\left(\frac{b}{m}\right) \right|^l \right]^\frac{1}{l}.
\end{align*}
If we choose $x = tb + (1-t)a$, analogously we obtain
\begin{align*}
&\int_a^b (x-a)^p(b-x)^qf(x)dx\\
\leq & (b-a)^{p+q+1}\left[\beta(p+1, q+1)\right]^{\frac{l-1}{l}} \\&
\times\left[ \beta(q+1, p+\alpha+1)|f(b)|^l +m[\beta(q+1, p+1)-\beta(q+1, p+\alpha+1)]\left|f\left(\frac{a}{m}\right) \right|^l \right]^\frac{1}{l},
\end{align*}
which completes the proof.
\end{proof}

\begin{corollary}
\label{co2.2} In Theorem \ref{th2.3}, if $p=q$, then \eqref{2.6} reduces to
\begin{align*}
&\int_a^b (x-a)^p(b-x)^pf(x)dx\nonumber\\
\leq & (b-a)^{2p+1}\left[\beta(p+1, p+1)\right]^{\frac{l-1}{l}}\nonumber\\
&\times\min\left\{\left[ \beta(p+\alpha+1, p+1)|f(a)|^l +m[\beta(p+1, p+1)-\beta(p+\alpha+1, p+1)]\left|f\left(\frac{b}{m}\right) \right|^l \right]^\frac{1}{l},\right.\nonumber\\
 & \left. \left[ \beta(p+1, p+\alpha+1)|f(b)|^l +m[\beta(p+1, p+1)-\beta(p+1, p+\alpha+1)]\left|f\left(\frac{a}{m}\right) \right|^l \right]^\frac{1}{l}
\right\}.
\end{align*}
\end{corollary}

\begin{corollary}
\label{co2.5} In Theorem \ref{th2.3}, if $\alpha =1,$ i.e., if $|f|^{l}$ is $m-$convex on $[a,b]$, then \eqref{2.6} reduces to
\begin{align*}
&\int_a^b (x-a)^p(b-x)^qf(x)dx\nonumber\\
\leq & (b-a)^{p+q+1}\left[\beta(p+1, q+1)\right]^{\frac{l-1}{l}}\min\left\{\left[ \beta(q+2, p+1)|f(a)|^l +m\beta(q+1, p+2)\left|f\left(\frac{b}{m}\right) \right|^l \right]^\frac{1}{l},\right.\nonumber\\
 & \left. \left[ \beta(q+1, p+2)|f(b)|^l +m\beta(q+2, p+1)\left|f\left(\frac{a}{m}\right) \right|^l \right]^\frac{1}{l}
\right\}.
\end{align*}
\end{corollary}

\begin{remark}
\label{re3} As a special case of Corollary \ref{co2.5} for $m =1,$ that is for $|f|^{l}$ be convex on $[a,b]$, we
get
\begin{align*}
&\int_a^b (x-a)^p(b-x)^qf(x)dx\nonumber\\
\leq & (b-a)^{p+q+1}\left[\beta(p+1, q+1)\right]^{\frac{l-1}{l}}\left[ \beta(q+2, p+1)|f(a)|^l + \beta(q+1, p+2)\left|f\left(b\right) \right|^l \right]^\frac{1}{l}.
\end{align*}
\end{remark}

\section{New integral inequalities for quasi-convex functions}\label{se3}

\begin{theorem}\label{th3.1}
Let $f: [a,b]\rightarrow\mathbb{R}$ be continuous on $[a,b]$ such that $f \in L([a,b])$, $0\le a<b<\infty$ and let $k>1$. If $|f|^{\frac{k}{k-1}}$ is quasi-convex on $[a,b]$,  for some fixed  $p, q > 0$, then
\begin{align}\label{3.1}
 \int_a^b (x-a)^p(b-x)^qf(x)dx 
\leq    (b-a)^{p+q+1} \left[\beta(kp+1, kq+1)\right]^{\frac{1}{k}}\left( \max \left\{ \left| f(a)\right| ^{\frac{k}{k-1}},
\left| f(b)\right| ^{\frac{k}{k-1}}\right\} \right) ^{\frac{k-1}{k}}.
\end{align}
\end{theorem}

\begin{proof}
By Lemma \ref{le1.1}, H\"older's inequality, the definition of Beta function and the fact that $|f|^{\frac{k}{k-1}}$ is quasi-convex on $[a,b]$, we have
\begin{align*}
&\int_a^b (x-a)^p(b-x)^qf(x)dx\\
\leq & (b-a)^{p+q+1}\left[\int_0^1 (1-t)^{kp}t^{kq}dt\right]^\frac{1}{k}
\left[\int_0^1 |f(ta+(1-t)b)|^{\frac{k}{k-1}} dt\right]^\frac{k-1}{k} \\
\leq & (b-a)^{p+q+1}\left[\beta(kq+1, kp+1)\right]^{\frac{1}{k}} \left[\int_0^1  \max \left\{ \left| f(a)\right| ^{\frac{k}{k-1}},
\left| f(b)\right| ^{\frac{k}{k-1}}\right\}dt \right]^\frac{k-1}{k}\\
= & (b-a)^{p+q+1}\left[\beta(kq+1, kp+1)\right]^{\frac{1}{k}} \left[ \max \left\{ \left| f(a)\right| ^{\frac{k}{k-1}},
\left| f(b)\right| ^{\frac{k}{k-1}}\right\}  \right]^\frac{k-1}{k},
\end{align*}
which completes the proof.
\end{proof}

\begin{corollary}
\label{co3.1} Let $f$ be as in Theorem \ref{th3.1}. Additionally, if

$(1)$ $f$ is increasing, then we have
\begin{align*}
 \int_a^b (x-a)^p(b-x)^qf(x)dx
\leq    (b-a)^{p+q+1} \left[\beta(kp+1, kq+1)\right]^{\frac{1}{k}}f(b).
\end{align*}

$(2)$ $f$ is decreasing, then we have
\begin{align*}
 \int_a^b (x-a)^p(b-x)^qf(x)dx
\leq  (b-a)^{p+q+1} \left[\beta(kp+1, kq+1)\right]^{\frac{1}{k}}f(a).
\end{align*}
\end{corollary}

\begin{theorem}\label{th3.2}
Let $f: [a,b]\rightarrow\mathbb{R}$ be continuous on $[a,b]$ such that $f \in L([a,b])$, $0\le a<b<\infty$ and let $l\ge 1$. If $|f|^{l}$ is quasi-convex on $[a,b]$,  for some fixed  $p, q > 0$, then
\begin{align}\label{3.2}
 \int_a^b (x-a)^p(b-x)^qf(x)dx
\leq  (b-a)^{p+q+1} \beta(p+1, q+1) \left( \max \left\{ \left| f(a)\right| ^l,
\left| f(b)\right| ^l\right\} \right) ^{\frac{1}{l}},
\end{align}
where $\beta (x, y)$ is the Euler Beta function.
\end{theorem}

\begin{proof}
By Lemma \ref{le1.1}, H\"older's inequality, the definition of Beta function and the fact that $|f|^{l}$ is quasi-convex on $[a,b]$, we have
\begin{align*}
&\int_a^b (x-a)^p(b-x)^qf(x)dx\\
=&(b-a)^{p+q+1} \int_0^1 \left[(1-t)^pt^q\right]^\frac{l-1}{l} \left[(1-t)^pt^q\right]^\frac{1}{l} f(ta+(1-t)b)dt\\
\leq & (b-a)^{p+q+1}\left[\int_0^1 (1-t)^{p}t^{q}dt\right]^\frac{l-1}{l}
\left[\int_0^1 (1-t)^{p}t^{q} |f(ta+(1-t)b)|^{l} dt\right]^\frac{1}{l} \\
\leq & (b-a)^{p+q+1}\left[\beta(q+1, p+1)\right]^{\frac{l-1}{l}}  \left[\max \left\{ \left| f(a)\right| ^l,
\left| f(b)\right| ^l\right\} \beta(q+1, p+1) \right]^\frac{1}{l}\\
=&(b-a)^{p+q+1} \beta(p+1, q+1) \left( \max \left\{ \left| f(a)\right| ^l,
\left| f(b)\right| ^l\right\} \right) ^{\frac{1}{l}},
\end{align*}
which completes the proof.
\end{proof}

\begin{corollary}
\label{co3.2} Let $f$ be as in Theorem \ref{th3.2}. Additionally, if

$(1)$ $f$ is increasing, then we have
\begin{align*}
 \int_a^b (x-a)^p(b-x)^qf(x)dx
\leq    (b-a)^{p+q+1} \beta(p+1, q+1)f(b).
\end{align*}

$(2)$ $f$ is decreasing, then we have
\begin{align*}
 \int_a^b (x-a)^p(b-x)^qf(x)dx
\leq  (b-a)^{p+q+1} \beta(p+1, q+1)f(a).
\end{align*}
\end{corollary}

\end{document}